\documentclass[11pt]{article}

\makeindex
\usepackage[all]{xy}
\usepackage{graphicx}

\usepackage{amsmath}
\usepackage{amsfonts}
\usepackage{amssymb}
\usepackage{amsthm}

\newcommand{\isom}{\cong}

\newcommand{\Z}{{\bf{Z}}}

\newcommand{\Q}{{\bf{Q}}}

\newcommand{\F}{{\bf{F}}}
\newcommand{\T}{{\bf{T}}}

\newcommand{\tensor}{\otimes}

\newcommand{\ra}{{\rightarrow}}

 1
\DeclareFontEncoding{OT2}{}{} 
  
\newcommand{\Fp}{{{\F}_p}}
\newcommand{\Fpbar}{{\overline{\F}}_p}

\newcommand{\pp}{{\mathfrak{p}}}

\newcommand{\OO}{{\mathcal{O}}}
\newcommand{\tO}{{\widetilde{\OO}}}
\newcommand{\tK}{{\widetilde{K}}}
\newcommand{\Of}{{\OO}}

\usepackage{bm}
\bmdefine\bmu{\mu}

\newcommand{\comment}[1]{}

\newtheorem{lem}{Lemma}[section]
\newtheorem{cor}[lem]{Corollary}
\newtheorem{prop}[lem]{Proposition}

\newtheorem{thm}[lem]{Theorem}

\theoremstyle{definition}

\newcommand{\thetitle}
{The index of a certain quotient of the Hecke algebra in its normalization}

\begin{document}
\parindent=2em

\title{\thetitle}
\author{Amod Agashe}

\date{}
\maketitle

\begin{abstract}
Let $\Gamma$ be a congruence subgroup of~${\rm SL}_2(\Z)$, 
and let $f$ 
be a normalized eigenform of weight~$k$ on~$\Gamma$. 
Let $K$ denote the number field 
generated over~$\Q$ by the Fourier coefficients of~$f$.
Let $R$ denote the 
the order in~$K$ generated 
by the Fourier coefficients of~$f$, which is contained
in the ring of integers~$\OO$ of~$K$.
We relate the primes that divide the index of~$R$ in~$\OO$
to primes~$p$ such that 
$f$ is congruent to a conjugate of~$f$ modulo a prime ideal 
of residue characteristic~$p$. The index mentioned above 
is the same as the index of the quotient of the Hecke algebra 
by the annihilator ideal of~$f$ in its normalization.
\end{abstract}



\section{Introduction and results}


Let $N$ be a positive integer, and as usual, let $\Gamma(N)$ denote the principal
congruence subgroup of level~$N$, i.e., the subgroup of~${\rm SL}_2(\Z)$
consisting of matrices that reduce to the identity matrix modulo~$N$.
Let $\Gamma$ be a subgroup of~${\rm SL}_2(\Z)$ 
that contains $\Gamma(N)$; so it is a congruence subgroup. 
Let $f$ 
be a cusp form of weight~$k$ on~$\Gamma$ that is an eigenform (i.e., an eigenvector
for the Hecke operators) and is normalized (i.e. whose first Fourier coefficient is~$1$).
Let $\OO$ denote the ring of integers of the number field~$K$
generated over~$\Q$ by the Fourier coefficients of~$f$. Let $R$ denote the order inside~$\OO$ 
generated by the Fourier coefficients of~$f$. 
In general, $R$ need not be all of~$\OO$, and so a natural question is: what
can one say about the index $[\Of:R]$? It is known to experts that 
the primes that divide this index are related to primes of congruence between~$f$
and its conjugates;
however, this does not seem to be documented in the
literature. The purpose of this article
is to state and prove the precise relations. 

The index alluded to above can be viewed in a different way, as follows.
Let $S_k(\Gamma)$ denote the space of cusp forms 
of weight~$k$ on~$\Gamma$. 
Let $\T$ denote the Hecke algebra generated by the Hecke operators 
acting on~$S_k(\Gamma)$. Then we have a map $\T \ra K$ obtained by taking~$T_n$ 
to the $n$-th Fourier coefficient~$a_n(f)$ of~$f$. 
The image of this map is precisely~$R$, and the kernel is~$I_f = {\rm Ann}_{\T}(f)$.
Thus we have an isomorphism $\T/I_f \isom R$. 
Now $\OO$ is the integral closure of~$R$ in~$K$, so 
the index $[\Of:R]$ is just the index of~$\T/I_f$ in its normalization.

Let $J_\pp$ denote the set of elements 
$\sigma \in {\rm Aut}(K/\Q)$ such that 
$\sigma(\pp) \subseteq \pp$
and $\sigma$ acts trivially on~$\OO/\pp$; it is a subgroup of~${\rm Aut}(K/\Q)$,
and if $K/\Q$ is Galois, then it is the 
usual inertia group, 
denoted~$I_\pp$\footnote{In the literature, the inertia group seems to be
defined only if the number field is Galois, and its relation to ramification
is discussed only in the Galois case. In the non-Galois case, the relation
may not be exactly the same, so to avoid any potential confusion,
we are not calling $J_\pp$ the inertia group.}. 
In the following, $p$ will always denote an arbitrary prime number.

\begin{prop} \label{prop1}
If $f$ is congruent to a conjugate of~$f$ by a nontrivial 
element of~${\rm Aut}(K/\Q)$
modulo a prime ideal~$\pp$ in~$\Of$ with residue characteristic~$p$, 
then either $J_\pp$ is nontrivial
or $p$~divides $[\Of:R]$.
If for 
some prime ideal~$\pp$ of~$\OO$, $J_\pp$ has a nontrivial element~$\sigma$, 
then 
$f$ is congruent modulo~$\pp$ to the conjugate of~$f$ by~$\sigma$.
\end{prop}
\begin{proof} The second statement of the proposition above holds since by hypothesis,
for all~$n$, $\sigma(a_n(f)) = a_n(f)$ in~$\OO/\pp$, i.e., 
$\sigma(a_n(f)) - a_n(f) \in \pp$, and so
$f \equiv \sigma(f)\bmod \pp$, as needed. 

We now prove the first statement.
Let $\sigma$ denote the 
element of~${\rm Aut}(K/\Q)$ in the hypothesis.
Since the Fourier coefficients of~$f$ generate~$R$,
and for all~$n$, $a_n(f) \equiv \sigma(a_n(f)) \bmod \pp$, 
\begin{eqnarray} \label{eqn0}
{\rm for\ all}\ x \in R, x - \sigma(x) \in \pp\ .
\end{eqnarray}
If $p$~divides $[\Of:R]$, then we are done, so suppose it doesn't.
Let $m = [\Of:R]$ and let $y \in \OO$. Then $my \in R$, so $m(y - \sigma(y)) =
(my) - \sigma(my) \in \pp$ by~(\ref{eqn0}).
So the image of~$m(y - \sigma(y))$
in $\OO/\pp$ is~$0$. 
Since $p$ is coprime to~$m$, this means that 
the image of~$y - \sigma(y)$ in~$\OO/\pp$ is~$0$. Thus
\begin{eqnarray} \label{eqn1}
{\rm for\ all}\ y \in \OO, y - \sigma(y) \in \pp. 
\end{eqnarray}
In
particular, if $y \in \pp$, then $\sigma(y) \in \pp$, so $\sigma(\pp) \subseteq \pp$.
Also, by~(\ref{eqn1}), $\sigma$ acts trivially on~$\OO/\pp$. 
Thus $\sigma \in J_\pp$, and so $J_\pp$ is nontrivial. 
\end{proof}

Note that the converse of the first statement in the proposition above need not be true
since ${\rm Aut}(K/\Q)$ could be trivial, while $[\Of:R]$ is nontrivial
(in particular, divisible by some prime~$p$).
To get some kind of a converse,
we pass to an extension of~$K$ to get enough automorphims.
Let $\tK$ be a Galois closure of~$K$ and $\tO$ the ring of integers of~$\tK$.

\begin{thm} \label{thm:main}
If 
$p$ ramifies in~$\tO$ or $p$~divides $[\tO:R]$, then
$f$ is congruent to a conjugate of~$f$ by a nontrivial 
element of~${\rm Gal}(\tK/\Q)$
modulo a prime ideal in~$\tO$ that has residue characteristic~$p$. 
\end{thm}

We will prove this theorem in the next section.

\begin{cor} \label{cor:main}
If 
$p$ ramifies in~$\OO$ or $p$~divides $[\Of:R]$, then
$f$~is congruent to a conjugate of~$f$ by a nontrivial element of~${\rm Gal}(\tK/\Q)$
modulo a prime ideal in~$\tO$  that has residue characteristic~$p$. 
\end{cor}
\begin{proof}
This follows from the theorem above
since if $p$ ramifies in~$\OO$, then it ramifies in~$\tO$, and
if $p$~divides $[\OO:R]$, then it also divides $[\tO:R]$. 
\end{proof}

The converse of the corollary above is not true if $K/\Q$~is not Galois, since
any nontrivial element in~${\rm Gal}(\tK/K)$
fixes~$f$, and there are primes~$p$ that do not ramify
in~$\OO$ and do not divide $[\Of:R]$ (since there are only finitely
many primes that ramify in~$\OO$ or divide $[\Of:R]$). 
Things are nicer if $K/\Q$~is Galois,
since then from the proposition and theorem above, we get 
an ``if and only if'' statement:

\begin{cor}
Suppose $K/\Q$ is Galois. Then 
$p$~ramifies in~$\OO$ or $p$~divides $[\Of:R]$ if and only if
$f$~is congruent to a conjugate of~$f$ by a nontrivial 
element of~${\rm Gal}(K/\Q)$
modulo a prime ideal in~$\OO$ that has residue characteristic~$p$. 
\end{cor}

In the rest of this section, we mention a couple of situations where
the index~$[\Of:R]$ is relevant. 
There are some other indices  considered in the literature that are 
analogous to the one in this article. 
Let $\Gamma = {\rm SL}_2(\Z)$. As explained in~\cite{jochnowitz:duke}, 
$\T \tensor \Q$ is a product of number fields and 
has a unique maximal order~$\OO_k$, which  is isomorphic to the
product of the rings of integers of its component number fields. 
In loc. cit., the author studies the index~$[\OO_k:\T]$ (in loc. cit., $\T$ is denoted~$\T_k$)
as the weight~$k$ varies. Now by projection, we get a map $\OO_k \ra \OO$, which is
a surjection. Moreover, the image of~$\T$ under this surjection 
is contained in~$R$. Thus we have a surjection $\OO_k/\T \ra \OO/R$,
and so our index $[\OO:R]$ divides the index~$[\OO_k:\T]$ considered in loc. cit.;
this latter index is just the index of~$\T$ in its normalization, since the 
normalization of~$\T$ is~$\OO_k$.
Similarly, for $\Gamma = \Gamma_1(N)$, in~\cite{calegari-emerton}, 
the authors consider the index of the image of~$\T$ acting on newforms
in its normalization; for reasons as above, our index divides this index as well.

Our index is also related to the discriminant of
the Hecke algebra: recall that if 
$M$ is a free $\Z$-module of rank~$n$ with basis
$b_1, \ldots, b_n$, 
then the discriminant of~$M$, denoted~${\rm disc}(M)$, is the determinant of the 
$n \times n$ matrix 
with $(i,j)$-th entry ${\rm tr}(b_i b_j)$, where ${\rm tr}(b)$,
for $b \in M$, denotes
the trace of left multiplication by~$b$.
For simplicity, we only work with $\Gamma = \Gamma_0(N)$
and $\T^{\rm new}$ (instead of~$\T$), where $\T^{\rm new}$
the quotient of~$\T$ that acts faithfully on the 
subspace~$S_k(\Gamma_0(N)^{\rm new}$ spanned by newforms. 
In what follows, we have adapted the discussion in~\cite[\S6]{eki:congr}. 
Let~$f_i$,  for $i$ in an appropriate indexing set~$I$, be a set of representatives
for the Galois orbits of newforms. For each~$i \in I$, let $\T_{f_i}$ denote the 
quotient of~$\T^{\rm new}$ that acts faithfully on the subspace spanned by~$f_i$
and its Galois conjugates. 
Then ${\T^{\rm new}}$ injects into
$\oplus_{i \in I} \T_{f_i}$. 
Also, for each~$i \in I$, letting $K_i$ denote 
the number field generated by the Fourier coefficients of~$f_i$, 
the map $\T \ra K_i$ that takes $T_n$ to the $n$-th Fourier coefficient of~$f_i$
lands in~$\OO_i$, the ring of integers of~$K_i$, has image an order~$R_i$ in~$K_i$,
and has kernel~$I_{f_i}$, giving an isomorhpism $\T_{f_i} = \T/I_{f_i} \isom R_i$.
Denoting the discriminant by~$\rm disc$, from the properties of the discriminant, it follows
that 
\begin{eqnarray}\label{eqndiff}
{\rm disc}(\T^{\rm new}) &  = & \bigg| \frac{\oplus_{i \in I} \T_{f_i}}{\T^{\rm new}} \bigg|^2 \cdot
\prod_{i \in I} {\rm disc}(\T_{f_i}) \nonumber\\
& = & \bigg| \frac{\oplus_{i \in I} \T_{f_i}}{\T^{\rm new}} \bigg|^2 \cdot
\prod_{i \in I} \Big( [\OO_i: R_i]^2 \cdot
{\rm disc} (\OO_i) \Big)
\end{eqnarray}
Now a prime~$p$ divides the term $| (\oplus_{i \in I} \T_{f_i})/\T^{\rm new}  |$ 
in~(\ref{eqndiff})
if and only 
some $f_i$ is congruent to a normalized eigenform~$g$ 
in~$S_k(\Gamma_0(N)^{\rm new}$ that is not conjugate to~$f_i$,
modulo a prime ideal lying over~$p$ in the ring of integers of 
the number field generated by the Fourier coefficients of~$f_i$ and~$g$
(e.g., see~\cite[p.193--194, 196]{ribet:modp}). 
From Corollary~\ref{cor:main}, we see that if a prime~$p$ divides
the remaining term~$\prod_{i \in I} ( [\OO_i: R_i]^2 \cdot
{\rm disc} (\OO_i) )$ in~(\ref{eqndiff}), then some 
$f_i$~is congruent to a conjugate of itself by a nontrivial element of
the Galois group of a Galois closure~$\tK_i$ of~$K_i$, 
modulo a prime ideal in the ring of integers of~$\tK_i$
that has residue characteristic $p$; from what we have seen above,
unfortunately we do not have an ``if and only if'', unlike in the previous sentence. 
We remark, in this context, that a prime~$p$ divides ${\rm disc}(\T^{\rm new})$
if and only if there is a 
congruence in characteristic~$p$ between two normalized eigenforms 
in~$S_k(\Gamma_0(N)^{\rm new}$ (see, e.g.,~\cite[Prop~2]{calegari-stein}); 
the decomposition~(\ref{eqndiff}) above shows the different roles 
that congruences between conjugate forms and 
congruences between
non-conjugate forms play in the discriminant.



The rest of this article is devoted to the proof of~Theorem~\ref{thm:main}.\\

\noindent {\em Acknowledgement}: The structure of the proof 
of Theorem~\ref{thm:main} was indicated to us by
K.~Ribet, for which we are very grateful. 
We would also like to thank H.~Hida for answering some questions
related to the topic of this article.

\section{Proof of Theorem~\ref{thm:main}}

Suppose $p$ ramifies in~$\tO$. Then for 
any prime ideal~$\pp$ of~$\tO$ lying over~$p$, the inertia group at~$\pp$ is nontrivial; 
let $\sigma$ be any nontrivial element of it. 
Then for all~$n$, $\sigma(a_n(f)) = a_n(f)$ in~$\tO/\pp$, i.e., 
$\sigma(a_n(f)) - a_n(f) \in \pp$, and so $f \equiv \sigma(f)\bmod \pp$, 
and we are done.
So suppose that $p$ does not ramify in~$\tO$. Assume that
$p$ divides $[\tO:R]$. It remains to show that then 
$f$ is congruent to a conjugate of~$f$ by a nontrivial element of~${\rm Gal}(\tK/\Q)$
modulo a prime ideal in $\tO$ that has residue characteristic $p$. 

Let $A$ be a finite semisimple commutative $\Fp$-algebra. 
We associate to it the set 
$S(A)$ that consists of the $\Fp$-algebra homomorphisms from~$A$ to~$\Fpbar$. 

\begin{lem}\label{lem1} 
The cardinality of $S(A)$ is the $\Fp$-dimension of A.
\end{lem}
\begin{proof}
A simple
commutative algebra has no nontrivial ideals, and is hence a field.
Thus $A$ is a product of fields that are finite and contain $\Fp$, i.e.,
finite extensions of $\Fp$.
So $A \isom \F_{p^{f_1}} \times \cdots \times  \F_{p^{f_r}}$ 
for some $f_1, \ldots, f_r$. 
Let $\{e_i\}_{i=1}^r$ denote the standard basis for this product and
let $\phi \in S(A)$. Then $0 = \phi(0) = \phi(e_i e_j) = \phi(e_i)
\phi(e_j)$. So for each $i \neq j$, if $\phi(e_i) \neq 0$, then $\phi(e_j) = 0$. Thus
only one $\phi(e_i)$ is non-zero. So the cardinality of $S(A)$ at most the product
of the cardinalities of the various $S(\F_{p^{f_i}})$, and clearly it is at least
as much, and so the two are equal.
The number of $\Fp$-algebra 
homomorphisms of $\F_{p^{f_i}}$ into $\Fpbar$ is the order of 
the Frobenius automorphism, i.e., ${f_i}$. So the cardinality of $S(A)$
is $f_1 \cdots f_r = \dim A$.
\end{proof}

Let $B$ be a subalgebra of~$A$. The following lemma is probably well known,
but we could not find a suitable reference.

\begin{lem} 
$B$ is semisimple. 
\end{lem} 
\begin{proof}
First, note that all the algebras in question are Artinian since they are finite. 
Now an Artinian ring is semisimple if and only if it has no nilpotent ideals other
than the zero ideal (e.g., see~\cite[Theorem~5.3.5]{cohn:algebra} or ``Structure theorem
for semi-primitive Artinian rings'' in~\cite[\S4.4]{jacobson:basicII}).
Suppose $B$ is not semisimple. Then 
it has a nilpotent ideal~$I$ other than the zero ideal. Hence there is a nonzero 
element~$b \in I$
that is nilpotent. Then the ideal generated by~$b$ in~$A$ is nilpotent (considering
that $A$ is commutative) and is nonzero (since it contains~$b$). 
Then from the second sentence of this proof, $A$ is semisimple. 
This contradiction proves the lemma.
\end{proof}

\begin{lem} \label{lem3}
If $B$ is a proper subalgebra of~$A$ (i.e., $B \neq A$), 
then the natural restriction map $S(A) \ra S(B)$ that takes~$\phi$ 
to~$\phi|_B$  is not injective.
\end{lem} 
\begin{proof} 
Considered as $\F_p$-vector spaces, $B$ is a proper subspace of~$A$, and so has
smaller dimension.
Then by Lemma~\ref{lem1}, 
$S(A)$ is bigger than $S(B)$. The lemma follows.
\end{proof}

Since $p$ does not ramify in~$\tO$, we
have  $p \tO = \pp_1 \cdots \pp_r$, for distinct prime ideals $\pp_1, \ldots, \pp_r$
of~$\tO$. 
Then 
\begin{eqnarray} \label{eqn2}
\tO/p \tO = \tO/(\pp_1 \cdots\pp_r) \isom \tO/\pp_1 \times \cdots \times \tO/\pp_r\ ,
\end{eqnarray}
where the second isomorphism comes from the Chinese remainder theorem. 
For each $i= 1, \ldots, r$, let $\pi_i$ denote the natural projection
map from $\tO/\pp_1 \times \cdots \times \tO/\pp_r$ to~$\tO/\pp_i$. 
Let $f$ denote the degree of $\tO/\pp_i$ over $\F_p$ (it is the same
for all~$i$, since $\tK/\Q$ is Galois). 
Then for each $i= 1, \ldots, r$, 
there are $f$ embeddings of~$\tO/\pp_i$ into $\Fpbar$, call them
$\tau_{i,1}, \ldots, \tau_{i,f}$. 
Thus we get a list of embeddings
of~$\tO/p \tO$ into~$\Fpbar$ 
by projecting onto one of the $r$ factors in~(\ref{eqn2}) and then composing with
one of the~$f$ embeddings into~$\Fpbar$ (more precisely, the list consists of
$\tau_{i,n} \circ \pi_i$ for $i = 1, \ldots, r$, $n = 1, \ldots, f$); 
there are $rf$ members in this list.

\begin{lem} \label{lem4}
The embeddings of~$\tO/p \tO$ into~$\Fpbar$
in the list above are distinct and 
$S(\tO/p \tO)$ consists of the embeddings in this list.
\end{lem}
\begin{proof}
By Lemma~\ref{lem1}, the  size of $S(\tO/p \tO)$ is its 
$\Fp$-dimension, which by~(\ref{eqn2}) is $rf$. So the second 
claim in the lemma follows from the first, which we now prove.
Suppose that for some $j,k \in \{1, \ldots, r\}$ and $n, m
\in \{1, \ldots, f\}$,
the embedding obtained by projecting to the $j$-th factor 
followed by the $n$-th possible embedding in~$\Fpbar$ is
the same as 
the embedding obtained by projecting to the $k$-th factor 
followed by the $m$-th possible embedding in~$\Fpbar$; we
will show that then $j=k$ and $m=n$, thus proving lemma.

Let $x_1, \ldots, x_r \in \tO$.
From the assumption above, 
we have $\tau_{j,m} \pi_j (x_1 + \pp_1, \ldots, x_r + \pp_r)
= \tau_{k,n} \pi_k (x_1 + \pp_1, \ldots, x_r + \pp_r)$, i.e.,
\begin{eqnarray} \label{eqn4}
\tau_{j,m} (x_j + \pp_j)
= \tau_{k,n} (x_k + \pp_k)\ .
\end{eqnarray}
If $j \neq k$, then pick $x_j =1$ and $x_k =0$ to conclude from
the equation above that 
$1=0$; this contradiction shows that $j=k$. 
If $\tO/\pp_j = \F_p$, then there is just a single~$\tau_{j,m}$, so $m=n=1$. 
If $\tO/\pp_j \neq \F_p$ and $m \neq n$, then 
$\tau_{j,m}$ and $\tau_{j,n}$ differ by a nontrivial power of the Frobenius,
which would give different values for any $x_j + \pp_j \in (\tO/\pp_j) \backslash \F_p$
on the two sides of~(\ref{eqn4}) (note that $j=k$ now);
this contradiction shows that $m=n$, and we are done.
\end{proof}

\begin{lem}\label{lem2} 
The natural map $R/pR \ra \tO/p \tO$ is
not surjective.
\end{lem} 
\begin{proof} 
Since $R$ and~$\tO$ are free $\Z$-modules of the same rank, 
$R/pR$ and~$\tO/p \tO$ are vector spaces over~$\Fp$ of the same dimension. 
So it suffices to show that the map $R/pR \ra \tO/p \tO$
is not injective, which is what we shall do.
Since $p$~divides $[\tO:R]$, 
there is an $x \in \tO \backslash R$ such
that $px\in R$.  
Suppose that the  image of~$px$ in~$R/pR$ is zero. 
Then $px = py$ for some~$y \in R$. So $p(x-y)=0$. Since $\tO$ is torsion-free,
we have $x = y$. But $x \not\in R$, while $y \in R$. This contradiction
shows that the  image of~$px$ in~$R/pR$ 
is nonzero; denote this image by~$z$. 
Now under the natural map $R/pR \ra \tO/p \tO$, 
$z$~maps to the image of~$px$ in~$\tO/ p \tO$; this image is~$0$. Thus 
under map $R/pR \ra \tO/p \tO$, the nonzero element~$z$ goes to~$0$, so 
this map is not injective, as was to be shown.
\end{proof}

Take $A = \tO/p \tO$ and take $B$ to be the image of $R/pR$ in $\tO/p \tO$
in the discussion above. So $B$ is smaller than $A$
by Lemma~\ref{lem2}.
Thus by Lemma~\ref{lem3}, two elements of~$S(A) = S(\tO/p \tO)$ map to the same
element of~$S(B)=S(R/p R)$.
Then by Lemma~\ref{lem4}, 
two of the embeddings of~$\tO/p \tO$ in~$\Fpbar$
constructed 
just before that lemma agree on $R/ p R$. 
The theorem essentially follows because one of these two embeddings can be obtained from the other
by applying an appropriate element of~${\rm Gal}(\tK/\Q)$, as we now explain in detail.

From what we said two sentences ago, for some $j,k \in \{1, \ldots, r\}$ and $n, m
\in \{1, \ldots, f\}$, for all $a \in R$, 
we have $\tau_{j,m} \pi_j (a + \pp_1, \ldots, a + \pp_r)
= \tau_{k,n} \pi_k (a + \pp_1, \ldots, a + \pp_r)$,
i.e., 
\begin{eqnarray} \label{eqn3}
\tau_{j,m} (a + \pp_j)
= \tau_{k,n} (a + \pp_k)\ .
\end{eqnarray}

Since $\tau_{j,m}$ and $\tau_{k,n}$ identify 
$\tO/\pp_j$ and $\tO/\pp_k$ respectively each with a copy of 
of~$\F_{p^f}$ in~$\Fpbar$, and there is only one such copy, call it~$F$,
we see that $\tau_{j,m}$ and $\tau_{k,n}$ induce
isomorphisms $\widetilde{\tau}_{j,m}: \tO/\pp_j  \stackrel{\isom}{\ra} F$
and $\widetilde{\tau}_{k,n}: \tO/\pp_k  \stackrel{\isom}{\ra} F$
such that $\widetilde{\tau}_{j,m} (a + \pp_j) = \widetilde{\tau}_{k,n} (a + \pp_k)$
(from equation~(\ref{eqn3})).
Then $\phi := \widetilde{\tau}_{k,n}^{-1} \circ \widetilde{\tau}_{j,m}$
is an isomorphism $\tO/\pp_j  \stackrel{\isom}{\ra} \tO/\pp_k$ 
such that 
\begin{eqnarray} \label{eqn5}
 \phi(a + \pp_j) = a + \pp_k\ . 
\end{eqnarray}

Pick a $\sigma \in {\rm Gal}(\tK/\Q)$ such that $\pp_k = \sigma (\pp_j)$
(we can do this since $\tK/\Q$ is Galois); it induces
an isomorphism $\tO/\pp_j \stackrel{\isom}{\ra} \tO/\pp_k$, which we will again denote by~$\sigma$.
Then 
\begin{eqnarray} \label{eqn6}
\forall b \in \tO, \sigma^{-1}(b + \pp_k) = \sigma^{-1}(b) + \pp_j\ . 
\end{eqnarray}

The automorphism $\sigma^{-1} \circ \phi$ of~$\tO/\pp_j$ satisfies
\begin{eqnarray}
(\sigma^{-1} \circ \phi)(a + \pp_j) & = &   \sigma^{-1}(\phi(a + \pp_j)) \nonumber \\
& = & \sigma^{-1}(a + \pp_k) 
{\rm\ \  (by\ equation~(\ref{eqn5}))} \nonumber \\
& = & \sigma^{-1}(a) + \pp_j 
{\rm\ \  (by\ equation~(\ref{eqn6}))} \label{eqn7}
\end{eqnarray}
This automorphism of~$\tO/\pp_j$ is induced by
by some 
$\tau \in {\rm Gal}(\tK/\Q)$ that
leaves~$\pp_j$ invariant, since the inertia group at~$\pp_j$ is trivial,
as $p$ does not ramify in~$\tO$.
Then 
\begin{eqnarray}
\tau(a + \pp_j)
& = & (\sigma^{-1} \circ \phi)(a + \pp_j)  \nonumber \\
& = & \sigma^{-1}(a) + \pp_j
{\rm\ \  (by\ equation~(\ref{eqn7}))} \label{eqn8}
\end{eqnarray}
Applying $\tau^{-1}$ to~(\ref{eqn8}) , we get
\begin{eqnarray}
a + \pp_j 
& = &\tau^{-1} (\sigma^{-1}(a) + \pp_j)  \nonumber \\
& = &\tau^{-1} (\sigma^{-1}(a)) + \pp_j \nonumber \\
& = & (\tau^{-1} \circ \sigma^{-1})(a) + \pp_j \ . \label{eqn9} 
\end{eqnarray}
Since for all~$n$, we have $a_n(f) \in R$, equation~(\ref{eqn9})
applies with $a = a_n(f)$. 
So we see that $f$ and 
and its conjugate $(\tau^{-1} \circ \sigma^{-1})(f)$ are congruent modulo~$\pp_j$, 
as was left to be shown.

\bibliographystyle{amsalpha}         

\bibliography{biblio}
\end{document}